\documentclass[final]{dmtcs-episciences}

\usepackage[utf8]{inputenc}
\usepackage{subfigure}
\usepackage{amsfonts,amsmath,amssymb}
\usepackage{euscript,color}
\usepackage{graphpap}
\usepackage{tikz}

\usepackage[round]{natbib}

\newtheorem{theorem}{Theorem}
\newtheorem{corollary}{Corollary}

\newtheorem{proposition}{Proposition}

\author[M.~Archibald, A.~Blecher, C.~Brennan, A.~Knopfmacher, S.~Wagner and M.~D.~Ward]{Margaret Archibald\affiliationmark{1}\thanks{This material is based upon work supported by the National Research Foundation under grant number 89147.}
\and Aubrey Blecher\affiliationmark{1}
\and Charlotte Brennan\affiliationmark{1}\thanks{This material is based upon work supported by the National Research Foundation under grant number 86329.}\\
\and Arnold Knopfmacher\affiliationmark{1}\thanks{This material is based upon work supported by the National Research Foundation under grant number 81021.}
\and Stephan Wagner\affiliationmark{2,3}\thanks{Supported by the Knut and Alice Wallenberg Foundation}
\and Mark Daniel Ward\affiliationmark{4}\thanks{M.D. Ward's research is supported by National Science Foundation (NSF) grants DMS-1246818, CCF-0939370, and OAC-2005632, by the Foundation for Food and Agriculture Research (FFAR) grant 534662, by the National Institute of Food and Agriculture (NIFA) grants 2019-67032-29077 and 2020-70003-32299, by the Society Of Actuaries grant 19111857, by Cummins Inc. grant 20067847, by Sandia National Laboratories, and by Gro Master.}}
  
\title[The number of distinct adjacent pairs in geometrically distributed words]{The number of distinct adjacent pairs in geometrically distributed words}
\affiliation{
  University of the Witwatersrand, Johannesburg, RSA\\
  Uppsala Universitet, Sweden\\
  Stellenbosch University, RSA\\
  Department of Statistics, Purdue University, USA}
\keywords{geometric random variables, pairs, generating function, asymptotics}

\received{2019-08-12} 
\revised{2020-08-04}  
\accepted{2020-11-10} 
\begin{document}
\publicationdetails{22}{2020}{4}{10}{5686}
\maketitle
\begin{abstract}
A sequence of geometric random variables of length $n$ is a sequence of $n$ independent and identically distributed geometric random variables ($\Gamma_1, \Gamma_2,  \dots, \Gamma_n$) where $\mathbb {P}(\Gamma_j=i)=pq^{i-1}$ for $1~\leq~j~\leq~n$ with $p+q=1.$
We study the number of distinct adjacent two letter patterns in such sequences. Initially we directly count the number of distinct pairs in words of short length. Because of the rapid growth of the number of word patterns we change our approach to this problem by obtaining an expression for the expected number of distinct pairs in words of length $n$. We also obtain the asymptotics for the expected number as $n \to \infty$.
\end{abstract}

\section{Introduction}
\label{sec:intro}
A sequence of geometric random variables of length $n$ is a sequence of $n$ independent and identically distributed geometric random variables ($\Gamma_1, \Gamma_2,  \dots, \Gamma_n$) where
\[\mathbb {P}(\Gamma_j=i)=pq^{i-1}\quad\textrm{for}\quad 1~\leq~j~\leq~n\quad\textrm{with}\quad p+q=1.\]
The smaller the value of $q$, the greater the prevalence of smaller numbers in the sequence.
So $p$ is the probability of obtaining the letter 1, where $0\le p\le 1$, and $pq^{i-1}$ is the probability of obtaining the letter $i$ in any given position of the word. In this paper a {\em word} will mean a {\em geometrically distributed random word}. This probabilistic model is the usual Bernoulli memoryless model with infinite alphabet and a geometrical distribution on letters.

The combinatorial interest in this topic can be seen in many recent references on sequences of geometric random variables, see for example \cite{ABBKP, akp06, CP, FP, KW, KP96, L, LPW, lopr06, LW,  Prod96, P}. 

Our study of pairs also has potential application in information theory relating to more efficient storage of digital information. The first \cite{gulomb} is related to the design of optimal codes for geometrically distributed variables. The second \cite{BCSV} considers optimal prefix codes for pairs of geometrically distributed variables. The third
\cite{MSW, WSS} consider a slightly different (two-sided) distribution for symbols which is used in applications.

In this paper, given a word of length $n$, we study the number of distinct adjacent two letter patterns (or pairs). This is a generalisation of the problem of distinct values previously studied in \cite{akp06} by two of the current authors and subsequently in \cite{LPW}. A pair is made up of two consecutive letters $\Gamma_i \Gamma_{i+1}$ for $i=1$ to $n-1$. So for example the word $\Gamma = 12412413$ is a sequence of geometric random variables of length $n=8$. It is made up with four distinct letters 1 to 4 and it has four distinct pairs namely: 12, 24, 41, and 13.

In Section \ref{DirCount} we count the expected number of distinct pairs in short words. Thereafter, in Section \ref{ExpNum}, we tackle the problem of the expected number of distinct pairs in a randomly generated word of length $n$. Here is an outline of our approach: We utilize the methodology of the paper by Bassino, Cl\'{e}ment and Nicod\`{e}me, see \cite{BCN}. The full strength of this paper is not needed, because (in this case) we are only studying ``reduced" sets of patterns, i.e., we are always analyzing patterns of length two. These methods are available in many other sources (for example the auto-correlation polynomial method used by Guibas and Odlyzko \cite{guod}), but the Bassino (et al.) paper has some nice examples that make the material accessible to people who are new to this methodology.

\section{A direct count for the number of distinct pairs in small words}\label{DirCount}
We illustrate a more intuitive approach in the case of small words. The patterns of length four in the first column of Table~1 correspond to partitions of the set $\{1,2,3,4\}$ into blocks. For example, $aabc$ corresponds to the set partition $12/3/4$. The number of set partitions  is enumerated by the well-known Bell numbers (see \cite{Stan}). In view of the rapid growth of the Bell numbers, this is not a practical approach for large $n$ but has the advantage of giving insight into the use of the Bell numbers  as a way of generating the solution to the problem.

\subsection{Words of length four}\label{length4}
There are precisely 15 patterns of length four, these are shown in the first column of Table 1.

\begin{table}[htp]\renewcommand{\arraystretch}{2.2}
\begin{center}
\begin {tabular}{|c|c|c|c|}\hline
The patterns&Number of&Probability of&Type\\
&distinct pairs&this pattern&\\\hline
$aaaa$&1&$\frac{p^4}{q^4}\sum_{a} q^{4a}$&A\\\hline
$abab$&2&$\frac{p^4}{q^4}\sum_{a,b} q^{2a+2b}$&B\\\hline
$aaab$&2&$\frac{p^4}{q^4}\sum_{a,b} q^{3a+b}$&C\\\hline
$abbb$&2&$\frac{p^4}{q^4}\sum_{a,b} q^{a+3b}$&C\\\hline
$abaa, aaba$&3&$\frac{p^4}{q^4}\sum_{a,b} q^{3a+b}$&C\\\hline
$aabb, abba$&3&$\frac{p^4}{q^4}\sum_{a,b} q^{2a+2b}$&B\\\hline
$aabc, abac, abca$&3&$\frac{p^4}{q^4}\sum_{a, b, c} q^{2a+b+c}$&D\\\hline
$abbc, abcb$&3&$\frac{p^4}{q^4}\sum_{a, b, c} q^{a+2b+c}$&D\\\hline
$abcc$&3&$\frac{p^4}{q^4}\sum_{a, b, c} q^{a+b+2c}$&D\\\hline
$abcd$&3&$\frac{p^4}{q^4}\sum_{a, b, c, d} q^{a+b+c+d}$&E\\\hline
\end{tabular}\vspace{0.5cm}\caption{Probabilities for all the patterns in a four letter word}
\end{center}
\end{table}
They correspond to the number of restricted  growth functions of length four. This is equal to the number of set partitions of length four for the set $\{1, 2, 3, 4\}$, i.e., the Bell number $B(4)=15$, see \cite{Stan}. The sum of the probabilities of obtaining these patterns is 1 as expected. In the last column, we have split the 15 words into five types: A to E. The type depends on the number of distinct letters making up the word and their multiplicity. For example words of type C are made up of two distinct letters: three letters of one kind and one of the other kind. Words with the same type occur with equal probability.

\textbf{Notation}: We write $\sum_{a,b,c,d,e}$ to mean the sum over all values of $a, b, c, d$ and $e$ where these values are all distinct from each other.

Now define $S_A(a)$, $S_B(a,b)$, $S_C(a,b)$, $S_D(a,b,c)$ and $S_E(a,b,c,d)$ to be the probability of occurrence of words of type A, B, C, D and E respectively. Thus
\begin{align*}
S_E(a,b,c,d):&=\frac{p^4}{q^4}\sum_{a, b, c, d} q^{a+b+c+d}=4!\,\frac{p^4}{q^4}\,\sum_{a=1}^\infty \sum_{b>a} \sum_{c>b} \sum_{d>c} q^{a+b+c+d}\\
&=\frac{24q^{6}}{(1+q)^2(1+q^2)(1+q+q^2)}
\end{align*}
For words of type $D$, we use the inclusion/exclusion principle to obtain the probability
\begin{align*}&S_D(a,b,c):=\frac{p^4}{q^4}\,\sum_{a, b, c} q^{2a+b+c}\\
&=\frac{p^4}{q^4}\left(\sum_{a=1}^\infty \sum_{b=1}^\infty \sum_{c=1}^\infty q^{2a+b+c}-\sum_{a=1}^\infty \sum_{b=1}^ \infty q^{2a+b+b}-\sum_{a=1}^\infty \sum_{b=1}^ \infty q^{2a+b+a}-\sum_{a=1}^\infty \sum_{c=1}^\infty q^{2a+a+c}+2 \sum_{a=1}^\infty q^{4a}\right)\\
&=\frac{2pq^3(1+2q+3q^2)}{(1+q)^2(1+q^2)(1+q+q^2)}.\end{align*}
Similarly
\begin{align}&S_B(a,b):=\frac{p^4}{q^4}\,\sum_{a, b} q^{2a+2b}=\frac{2p^2q^2}{(1+q)^2(1+q^2)}\label{SB}\\
&S_C(a,b):=\frac{p^4}{q^4}\,\sum_{a, b} q^{3a+b}=\frac{p^2q(1+q+2q^2)}{(1+q)(1+q^2)(1+q+q^2)}\label{SC}\end{align}
and
\begin{align*}&S_A(a):=\frac{p^4}{q^4}\,\sum_{a\geq 1} q^{4a}=\frac{p^3}{(1+q)(1+q^2)}.\end{align*}
Let the probability of obtaining $\ell$ distinct pairs in a random word of length four be $P_4(\ell)$. By reading off from Table 1 we have
\[P_4(1)=S_A,\]
\[P_4(2)=S_B+2S_C,\]
\[P_4(3)=2S_B+2S_C+6S_D+S_E.\]

By multiplying the number of distinct pairs (the numbers in column two) and the corresponding probabilities (the expressions in column three) and summing  over all cases, we obtain the expected number of distinct pairs in four letter words. We will denote this by $E(4)$. Thus
\begin{align}
E(4)&=S_A+2(S_B+2S_C)+3(2S_B+2S_C+6S_D+S_E)\notag\\
&=\frac{1+9q+15q^2+20q^3+17q^4+11q^5-q^6}{(1+q)^2(1+q^2)(1+q+q^2)}\label{F4}.
\end{align}

A plot of this expression can be seen in Figure~1.

\begin{figure}
\begin{center}
\includegraphics[width=0.8\textwidth]{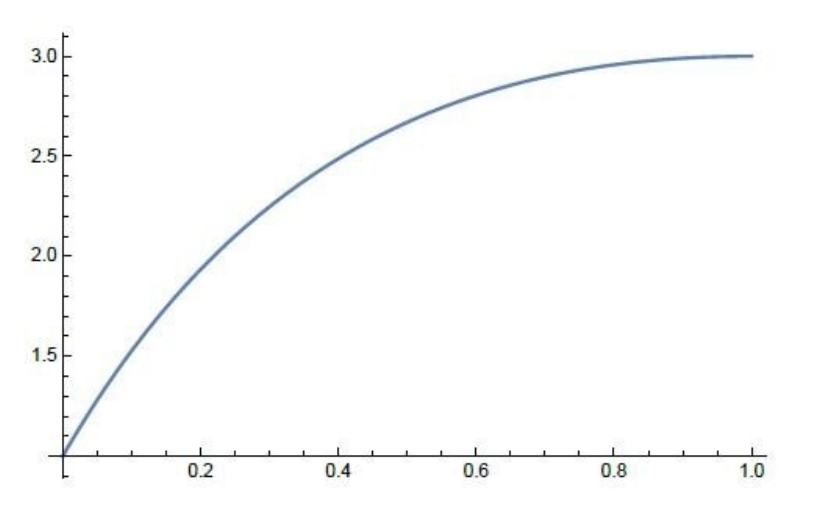}
\caption{Plot of $E(4)$ for $q=0$ to $q=1$.}\label{E4pic}
\end{center}
\end{figure}

\section{The expected number of distinct pairs}\label{ExpNum}
In this section, we use the methodology and notation from \cite{BCN}. An interested reader may wish to read this paper for clarity on the concepts used. We find an expression for the probability of obtaining a distinct pair. We split this problem into two cases: i) the case $ii$ of two equal letters, and ii) the $ij$ case of two distinct letters where $i \neq j$.

Let $P_{i,i}(n)$ and $P_{i,j}(n)$ be the probabilities of getting words of length $n$ which have a pair $ii$ and $ij$ (where $i \neq j$) respectively somewhere in position $1$ up to position $n-1$.

We also use $A(z)$ to denote the generating function for the alphabet of a set of words, for geometric words the probabilities sum up to one, so $A(z)$ is just $z$.

We have the following results:
\subsection{Analysis of distinct (adjacent) two letter pattern $ij$ with $i \neq j$}
We fix $i \neq j$ and analyze the occurrences of the pattern $ij$; then the only ``cluster" (to use the terminology in \cite{BCN}) is $ij$ itself. So the generating function $\xi(z,t)$ of the clusters $C$ becomes (see (6) in \cite{BCN}):
\[\xi(z,t)=pq^{i-1}pq^{j-1}tz^2.\]
Hence the generating function of the {\em decorated texts} is
\[T(z,t)=\frac{1}{1-A(z)-\xi(z,t)}=\frac{1}{1-z-pq^{i-1}pq^{j-1}tz^2},\]
and therefore using the  analytic inclusion-exclusion principle, the probability generating function of occurrences of $ij$ (with $z$ marking the length of the words, and $x$ marking the number of occurrences) is
\[F(z,x)=T(z,x-1)=\frac{1}{1-z-pq^{i-1}pq^{j-1}(x-1)z^2}.\]
It follows that the probability generating function of the words with {\em zero occurrences} of the pattern $ij$ is the $x^0$ term (which can be extracted by substituting $x=0$), and we obtain
\[[x^0]F(z,x)=F(z,0)=\frac{1}{1-z+pq^{i-1}pq^{j-1}z^2}.\]
So, finally the probability generating function of the words with {\em at least one occurrence} of $ij$ is
\begin{equation} \frac{1}{1-z}-\frac{1}{1-z+pq^{i-1}pq^{j-1}z^2}.\label{at least 1 ij}\end{equation}

\subsection{Analysis of distinct (adjacent) two letter pattern $ii$}
Now suppose that, instead, we fix $i$ and we analyze the occurrences of the pattern $ii$. The clusters all have the form $ii\cdots i$, i.e., they are all words that consist of two or more consecutive occurrences of $i$. So $\xi(z,t)$ of the set of clusters $C$ becomes
\[\xi(z,t)=\frac{pq^{i-1}pq^{i-1}tz^2}{1-pq^{i-1}tz}\]
and then $T(z,t)=\frac{1}{1-A(z)-\xi(z,t)}$ and the analysis is exactly the same as above, leading us (finally) to see that the probability generating function of the words with {\em at least one occurrence} of $ii$ is \
\begin{equation} \frac{1}{1-z}-\frac{1}{1-z+\frac{pq^{i-1}pq^{i-1}z^2}{1+pq^{i-1}z}}.\label{at least 1 ii}\end{equation}

\subsection{The expected number of distinct (adjacent) two letter patterns}
Now we garner the benefits of this analysis.
\begin{theorem}\label{Th1}
The generating function of the expected number of distinct (adjacent) two letter patterns is then
\[ Q(z):=\sum_{i=1}^\infty \sum_{j \neq i} \left( \frac{1}{1-z}-\frac{1}{1-z+pq^{i-1}pq^{j-1}z^2}\right)+\sum_{i=1}^\infty \left( \frac{1}{1-z}-\frac{1}{1-z+\frac{pq^{i-1}pq^{i-1}z^2}{1+pq^{i-1}z}}\right).\]
\end{theorem}
For example, the result~\eqref{F4} is a natural corollary, extracting the coefficients of $z^4$:
\[ [z^4] Q(z)=\frac{1+9q+15q^2+20q^3+17q^4+11q^5-q^6}{(1+q)^2(1+q^2)(1+q+q^2)}.\]

We now prove general formulae for the probabilities $P_{i,i}(n)$ and $P_{i,j}(n)$. To simplify the equations we will use the following notation: $a_i:=pq^{i-1}$, $b_i:=\sqrt{1+2a_i-3a_i^2}$, $c_{ij}:=p^2q^{i+j-2}$, and $d_{ij}:=\sqrt{1-4c}$.

\begin{proposition}\label{prop1}The probability of having two adjacent identical letters $ii$ is
\begin{equation*}
P_{i,i}(n)=1+\frac{2^{-n-1} \left(\left(1-b_i+a_i\right) \left(1-b_i-a_i\right)^n-\left(1+b_i+a_i\right) \left(1+b_i-a_i\right)^n\right)}{b_i}.
\end{equation*}
\end{proposition}

\begin{proposition}\label{prop2}The probability of having two adjacent distinct letters $ij$ is
\[P_{i,j}(n)=1+\frac{1-d_{ij}}{2d_{ij}}\left( \frac{2c_{ij}}{1+d_{ij}}\right)^n-\frac{1+d_{ij}}{2d_{ij}}\left( \frac{2c_{ij}}{1-d_{ij}}\right)^n.\]
\end{proposition}
\begin{proof}
We prove both Propositions \ref{prop1} and \ref{prop2}. From \eqref{at least 1 ii}, we see that the probability generating function of the words with {\em at least one occurrence} of $ii$ is
\[\frac{1}{1-z}-\frac{1}{1-z+\frac{pq^{i-1}pq^{i-1}z^2}{1+pq^{i-1}z}},\]
or, using the notation $a_i:=pq^{i-1}$, and $b_i:=\sqrt{1+2a-3a^2}$,  and using partial fractions, this becomes
\[\frac{1}{1-z}+\frac{(1-b_i+a_i)/(2b_i)}{1-z\frac{1-b_i-a_i}{2}}-\frac{(1+b_i+a_i)/(2b_i)}{1-z\frac{1+b_i-a_i}{2}}.\]
Hence, by extracting the coefficient of $z^n$ in the above expression we obtain

\[ P_{i,i}(n)=1+\frac{1-b_i+a_i}{2b_i} \left( \frac{1-b_i-a_i}{2}\right)^n-\frac{1+b_i+a_i}{2b_i}\left( \frac{1+b_i-a_i}{2}\right)^n.\]

Recall, from \eqref{at least 1 ij}, the probability generating function of the words with {\em at least one occurrence} of $ij$ is
\[\frac{1}{1-z}-\frac{1}{1-z+pq^{i-1}pq^{j-1}z^2},\]
or, using the notation $c_{ij}:=p^2q^{i+j-2}$, and $d_{ij}:=\sqrt{1-4c}$,  and using partial fractions, this becomes
\[\frac{1}{1-z}-\frac{1}{1-z+c_{ij}z^2}=\frac{1}{1-z}+\frac{(1-d_{ij})/(2d_{ij})}{1-2c_{ij}z/(1+d_{ij})} - \frac{(1+d_{ij})/(2d_{ij})}{1-2c_{ij}z/(1-d_{ij})}.\]
So again, extracting the coefficient of $z^n$ yields
\[P_{i,j}(n)=1+\frac{1-d_{ij}}{2d_{ij}}\left( \frac{2c_{ij}}{1+d_{ij}}\right)^n-\frac{1+d_{ij}}{2d_{ij}}\left( \frac{2c_{ij}}{1-d_{ij}}\right)^n.\]

\end{proof}

Denote the expected number of distinct pairs in a word of length $n$ by $E(n)$. It is obtained by summing $P_{i,i}(n)$ over all $i$ and $P_{i,j}(n)$ over all $i$ and $j$ for $i \neq j$. Thus
\[E(n)=\sum_{i=1}^\infty P_{i,i}(n) +\sum_{i,j\, \textrm{distinct}} P_{i,j}(n).\]
We write $S(n)$ for the first sum, and $T(n)$ for the second. In the following sections, we derive both exact and asymptotic formulae for these two quantities.

\section{Counting pairs arising from adjacent equal letters}\label{CountEqual}
In this section we look more closely at the $ii$ case first seen in Section \ref{ExpNum}.
As before $a_i:=(1-q)q^{i-1}$ and $b_i:=\sqrt{1+2a_i-3a_i^2}$ and the complete sum is
\begin{equation}S(n) =\sum_{i=1}^\infty \left( 1+\frac{(1-a_i-b_i)^n(1+a_i-b_i)-(1-a_i+b_i)^n(1+a_i+b_i)}{b_i 2^{n+1}}\right).\label{Sn1}\end{equation}
Our aim is to express $S(n)$, for general $n$, as a rational function of $q$, as done in \eqref{F4} for the special case $n=4$.

For a fixed $i$, the expected number of two equal adjacent letters $ii$ is
\begin{align}
&[z^n]\left(\frac{1}{1-z}-\frac{1}{1-z+\frac{(pq^{i-1}z)^2}{1+pq^{i-1}z}}\right)\notag\\
&=[z^n]\left( \frac{1}{1-z} - \frac{1+pq^{i-1}z}{1-(1-pq^{i-1})z(1+pq^{i-1}z)}  \right)\notag\\
&=[z^n]\left(\frac{1}{1-z}-\sum_{s=0}^\infty  (1-pq^{i-1})^sz^s(1+pq^{i-1}z)^{s+1}  \right)\label{coeff ii}
\end{align}
where
\[ [z^n]\sum_{s=0}^\infty  (1-pq^{i-1})^sz^s(1+pq^{i-1}z)^{s+1}=\sum_{s=0}^n (1-pq^{i-1})^s {s+1 \choose n-s}(pq^{i-1})^{n-s}.\]
Thus \eqref{coeff ii} becomes

\[ 1-\sum_{s=0}^n (1-pq^{i-1})^s {s+1 \choose n-s}(pq^{i-1})^{n-s} = 1 - \sum_{s=0}^n \sum_{k=0}^s (-1)^{s-k} {s \choose k}{s+1 \choose n-s} (pq^{i-1})^{n-k}.\]

We interchange the order of summation and then substitute $n-k = r$ and $n-s = t$ to obtain
\[1 - \sum_{k=0}^n \sum_{s=k}^n (-1)^{s-k} {s \choose k}{s+1 \choose n-s} (pq^{i-1})^{n-k}
= 1 - \sum_{r=0}^n \sum_{t=0}^r (-1)^{r-t} {n-t \choose n-r}{n-t+1 \choose t} (pq^{i-1})^r.\]
When $r=0$, the inner sum evaluates to $1$, so this reduces to
\[\sum_{r=1}^n \sum_{t=0}^r (-1)^{r-t+1} {n-t \choose n-r}{n-t+1 \choose t} (pq^{i-1})^r.\]
These calculations were done for a fixed $i$, thus we need to sum over all $i$'s, resulting in
\[\sum_{i=1}^\infty \sum_{r=1}^n \sum_{t=0}^r (-1)^{r-t+1} {n-t \choose n-r}{n-t+1 \choose t} (pq^{i-1})^r.\]
By evaluating the geometric sum we obtain
\begin{proposition}
The expected number of distinct pairs consisting of two equal letters in geometric words of length $n$ is
\[\sum_{r=1}^n \sum_{t=0}^r (-1)^{r-t+1} {n-t \choose n-r}{n-t+1 \choose t} \frac{p^r}{1-q^r}.\]
\end{proposition}

\section{Asymptotics for the number of distinct pairs with equal adjacent letters}\label{Asympequalletters}

In this section we determine an asymptotic expression for the average number of distinct pairs with equal adjacent letters for words of length $n$ as $n \to \infty$. We have the following result:
\begin{theorem}\label{Th2}
Let $q \in (0,1)$ be fixed. The average number of distinct pairs that consist of adjacent equal letters is equal to
\[\frac 12\log_{1/q}n+\frac 12-\frac{\gamma+2 \log p}{2 \log q}+\frac{1}{2 \log q}\sum_{k \in \mathbb{Z} \setminus \{0\}} \Gamma \Big(\frac{k\pi i}{\log q}\Big)e^{k \pi i \log_{1/q} (p^2 n)} + o(1)\]
as $n \to \infty$, where $\gamma$ is the Euler-Mascheroni constant and $\Gamma$ denotes the gamma function. The infinite sum is the Fourier series of a periodic function in $\log_{1/q} n$ with period $2$.
\end{theorem}

\begin{proof}
We use the same expressions for $a_i$ and $b_i$ as in the previous sections, i.e., $a_i:=pq^{i-1}$ and $b_i :=\sqrt{1+2a_i-3a_i^2}$. Observe that $a_i \to 0$ and $b_i \to 1$ as $i \to \infty$.
Recall from~\eqref{Sn1} that the average in question is equal to
\[S(n) = \sum_{i=1}^\infty \left( 1+\frac{(1-a_i-b_i)^n(1+a_i-b_i)-(1-a_i+b_i)^n(1+a_i+b_i)}{b_i 2^{n+1}}\right).\]
Our analysis of this sum proceeds in three steps:

\medskip

\textbf{Step 1.} We have $1-a_i-b_i \to 0$ as $i \to \infty$, rendering the terms involving $(1-a_i-b_i)^n$ asymptotically negligible. To make this precise, note first that $-2a \leq 1 - a - \sqrt{1+2a-3a^2} \leq 0$ for every $a \in [0,1]$ (as one easily verifies). It follows that
\[|1 - a_i - b_i| = \Big| 1 - a_i - \sqrt{1+2a_i-3a_i^2} \Big| \leq |2a_i|,\]
hence
\[\sum_{i=1}^\infty \frac{(1-a_i-b_i)^n(1+a_i-b_i)}{b_i 2^{n+1}} = O \Big( \sum_{i=1}^{\infty} a_i^n \Big).\]
The geometric sum evaluates to
\[\sum_{i=1}^{\infty} a_i^n = \frac{p^n}{1-q^n},\]
which goes to $0$ at an exponential rate. We conclude that the terms involving $(1-a_i-b_i)^n$ are indeed exponentially small and can be ignored in the following.

\medskip

\textbf{Step 2.} Next we approximate the remaining terms by simpler expressions. To this end, we will split the sum in an appropriate way. Observe first that $\frac{1-a+\sqrt{1+2a-3a^2}}{2}$ is a decreasing function of $a$ for $a \in [0,1]$, which implies that
\[\frac{1-a_1+b_1}{2} \leq \frac{1-a_2+b_2}{2} \leq \cdots \leq 1.\]
Next we note that
\[\frac{1 - a + \sqrt{1+2a-3a^2}}{2} = 1 - a^2 + O(a^3)\]
as $a \to 0$, hence
\[\frac{1-a_i+b_i}{2} = 1 - a_i^2 + O(a_i^3).\]
Moreover,
\[\frac{1 + a + \sqrt{1+2a-3a^2}}{2\sqrt{1+2a-3a^2}} = 1 + O(a^2),\]
and thus
\[\frac{1+a_i+b_i}{2b_i} = 1 + O(a_i^2).\]
Now we split the sum
\[\sum_{i=1}^{\infty} \Big( 1-  \frac{(1-a_i+b_i)^n(1+a_i+b_i)}{b_i 2^{n+1}}\Big) = \sum_{i=1}^{\infty} \Big( 1 - \frac{1+a_i+b_i}{2b_i} \Big( \frac{1-a_i+b_i}{2} \Big)^n \Big)\]
into the part where $i \leq i_0 = \lfloor \frac25 \log_{1/q} n \rfloor$ and the rest where $i > i_0$. The choice is made so that $n a_{i_0}^2 \to \infty$, while $n a_{i_0}^3 \to 0$ as $n \to \infty$. For $i > i_0$, we have
\begin{align*}
\frac{1+a_i+b_i}{2b_i} \Big( \frac{1-a_i+b_i}{2} \Big)^n &= \big( 1 + O(a_i^2) \big) \big( 1 - a_i^2 + O(a_i^3) \big)^n
\\
&= \exp \big( {- a_i^2 n} + O(a_i^3 n + a_i^2) \big) \\
&= \exp \big( {- a_i^2 n} \big) \big( 1 + O(a_i^3 n + a_i^2) \big).
\end{align*}
Hence
\[\sum_{i > i_0} \Big( 1 - \frac{1+a_i+b_i}{2b_i} \Big( \frac{1-a_i+b_i}{2} \Big)^n \Big) =
\sum_{i > i_0} \Big( 1 - \exp \big( {-a_i^2 n} \big) \Big) + O \Big( \sum_{i > i_0} (a_i^3 n + a_i^2) \Big).\]
Evaluating the geometric sums in the $O$-term, we find that it is $O( n^{-1/5})$. 

On the other hand, for $i \leq i_0$,
\begin{align*}
\frac{1-a_i+b_i}{2} &\leq \frac{1-a_{i_0}+b_{i_0}}{2} = 1 - a_{i_0}^2 + O(a_{i_0}^3) \\
&= 1 - \frac{p^2}{q^2} q^{2 \lfloor \frac25 \log_{1/q} n \rfloor} + O\big(q^{3 \cdot \frac25 \log_{1/q} n}\big) \\
&\leq 1 - \frac{p^2}{q^2} n^{-4/5} + O\big(n^{-6/5}\big),
\end{align*}
thus
\[\Big( \frac{1-a_i+b_i}{2} \Big)^n = O \Big( \exp \Big( - \frac{p^2}{q^2} n^{1/5} \Big) \Big),\]
which means that the sum
\[\sum_{i \leq i_0} \frac{1+a_i+b_i}{2b_i} \Big( \frac{1-a_i+b_i}{2} \Big)^n \]
goes to $0$ faster than any power of $n$ and is thus negligible. We replace it by
\[\sum_{i \leq i_0} \exp \big({-a_i^2 n} \big),\]
which is equally negligible by the same reasoning. In conclusion, we have found that
\[\sum_{i=1}^{\infty} \Big( 1-  \frac{(1-a_i+b_i)^n(1+a_i+b_i)}{b_i 2^{n+1}}\Big) = \sum_{i=1}^{\infty} \Big( 1 - \exp \big({-a_i^2 n} \big) \Big) + O\big(n^{-1/5}\big),\]
or combined with Step 1, that
\[S(n) = \sum_{i=1}^{\infty} \Big( 1 - \exp \big({-a_i^2 n}\big) \Big) + O\big(n^{-1/5}\big).\]

\medskip

\textbf{Step 3.} We are now left with the sum
\[\sum_{i=1}^{\infty} \Big( 1 - \exp \big({-a_i^2 n}\big) \Big)  = \sum_{i=1}^\infty \Big(1- \exp \big( {-p^2q^{2i-2} n} \big)\Big),\]
which is a special case of the harmonic sum
\[\sum_{i=0}^{\infty} \left(1 - e^{- b^{-i} x} \right).\]
The asymptotic behaviour of this harmonic sum as $x \to \infty$ is well known. The formula
\[\sum_{i=0}^{\infty} \left(1 - e^{- b^{-i} x} \right)
= \log_b x + \frac{\gamma}{\log b} + \frac12 - \frac{1}{\log b} \sum_{k \in \setminus \{0\}} \Gamma \Big( {- \frac{2k \pi i}{\log b}} \Big) e^{2k\pi i \log_b x} + O \Big( \frac1x \Big)
\]
is given for example in \cite{PW}, see also \cite{FGD}. Here, we can apply this formula with $b = \frac1{q^2}$ and $x = p^2 n$. Combining all contributions now gives us the formula in Theorem~\ref{Th2}.

\end{proof}

\section{Counting pairs arising from adjacent non-equal letters}\label{CountNonEqual}
Here we find the average number of pairs arising from adjacent non-equal letters in words of length $n$. We have
\begin{proposition}\label{prop4}
The expected number of distinct pairs consisting of two non-equal letters in geometric words of length $n$ is
\[T(n)=2\sum_{k=1}^{\lfloor n/2 \rfloor} \frac{(-1)^{k-1}p^{2k}q^k{n-k \choose k}}{(1-q^k)^2(1+q^k)}.\]
\end{proposition}
\begin{proof} From Theorem~\ref{Th1}, the expected number of distinct adjacent unequal letters is given by
\begin{align*}
&[z^n]\sum_{i=1}^\infty \sum_{j \neq i} \left(\frac{1}{1-z} - \frac{1}{1-z+pq^{i-1}pq^{j-1}z^2}\right),
\end{align*}
where
\begin{align*}
&[z^n]\frac{1}{1-z+pq^{i-1}pq^{j-1}z^2}=[z^n]\sum_{s=0}^\infty z^s(1-pq^{i-1}pq^{j-1}z)^s\\
&=\sum_{s=0}^\infty \sum_{k=0}^s [z^{n-s}] {s \choose k}(-p^2q^{i+j-2}z)^k=\sum_{s=0}^n {s \choose n-s} (-p^2q^{i+j-2})^{n-s}.
\end{align*}
Thus summing on $i$
\begin{align*}
&[z^n]\sum_{i=1}^\infty \sum_{j \neq i} \left(\frac{1}{1-z} - \frac{1}{1-z+pq^{i-1}pq^{j-1}z^2}\right)\\
&=\sum_{i=1}^\infty \sum_{j \neq i}\left[1-\sum_{s=0}^n {s \choose n-s} (-p^2q^{i+j-2})^{n-s}\right]\\
&=\sum_{i=1}^\infty \sum_{j \neq i}\sum_{s=0}^{n-1} {s \choose n-s}(-1)^{n-s+1} (p^2q^{i+j-2})^{n-s}\\
&=\sum_{s=0}^{n-1} \left(\frac{p^2}{q^2}\right)^{n-s}{s \choose n-s} (-1)^{n-s+1} \sum_{i=1}^\infty \left[\sum_{j=1}^{i-1} (q^{i+j})^{n-s}+\sum_{j=i+1}^\infty (q^{i+j})^{n-s}\right]\\
&=2\sum_{s=0}^{n-1} \left(\frac{p^2}{q^2}\right)^{n-s}{s \choose n-s} (-1)^{n-s+1} \sum_{i=1}^\infty \sum_{j=i+1}^\infty (q^{i+j})^{n-s},
\end{align*}
where $\sum_{i=1}^\infty \sum_{j \neq i}=2\sum_{i=1}^\infty \sum_{j>i}$. So finally
\begin{align*}
&[z^n]\sum_{i=1}^\infty \sum_{j \neq i} \left(\frac{1}{1-z} - \frac{1}{1-z+pq^{i-1}pq^{j-1}z^2}\right)\\
&=2\sum_{s=0}^{n-1} {s \choose n-s} (-1)^{n-s+1} \left(\frac{p^2}{q^2}\right)^{n-s}\frac{1}{1-q^{n-s}}\sum_{i=1}^\infty (q^{n-s})^{2i+1}\\
&=2\sum_{s=0}^{n-1} {s \choose n-s} (-1)^{n-s+1} \left(\frac{p^2}{q^2}\right)^{n-s}\frac{q^{3(n-s)}}{(1-q^{n-s})(1-q^{2(n-s)})}.
\end{align*}
Now, by replacing $s$ by $n-k$ we have the required result
\[T(n)=2\sum_{k=1}^{\lfloor n/2 \rfloor} \frac{(-1)^{k-1}p^{2k}q^k{n-k \choose k}}{(1-q^k)^2(1+q^k)}.\]
\end{proof}
\section{Asymptotics for the number of distinct pairs with non-equal adjacent letters}

We finally also consider the asymptotic behaviour of the sum $T(n)$ for which an explicit formula was derived in the previous section. Specifically, we prove the following result:
\begin{theorem}\label{Th3}
Let $q \in (0,1)$ be fixed. The average number of distinct pairs with non-equal adjacent letters is equal to
\begin{align*}
&\frac{(\log_{1/q} n)^2}2 - \Big( \frac{\gamma + 2\log p}{\log q} - \frac12 \Big) \log_{1/q} n + \frac{6\gamma^2 + \pi^2 + 24\gamma \log p + 24(\log p)^2}{12 (\log q)^2} - \frac{\gamma + 2\log p}{2\log q} - \frac1{12} \\
&+ \sum_{k \in \mathbb{Z} \setminus \{0\}} \Big({\frac{\Gamma(\chi_k)\log_{1/q} n}{\log q}} - \frac{\Gamma(\chi_k)(4\log p - \log q) - 2\Gamma'(\chi_k)}{2(\log q)^2} \Big) e^{2k \pi i \log_{1/q}(p^2 n)} \\ 
&- \sum_{m \in \mathbb{Z}} \frac{\Gamma(\hat{\chi}_m)}{2\log q} e^{(2m+1) \pi i \log_{1/q}(p^2 n)} + o(1)
\end{align*}
as $n \to \infty$, where $\chi_k = \frac{2k \pi i}{\log q}$ and $\hat{\chi}_m = \frac{(2m+1)\pi i}{\log q}$. 
\end{theorem}
We remark that the infinite sums can be interpreted as Fourier series again. See also Corollary~\ref{combined}.

\begin{proof}
We start with the expression for $P_{i,j}(n)$ obtained from Proposition \ref{prop2} in Section \ref{ExpNum}:
\[P_{i,j}(n) = 1+\frac{1-d_{ij}}{2d_{ij}}\left( \frac{2c_{ij}}{1+d_{ij}}\right)^n-\frac{1+d_{ij}}{2d_{ij}}\left( \frac{2c_{ij}}{1-d_{ij}}\right)^n,\]
where $c_{ij} = p^2 q^{i+j-2}$ and $d_{ij} = \sqrt{1-4c_{ij}}$. Note that $c_{ij} \to 0$ and $d_{ij} \to 1$ as $i+j \to \infty$. By symmetry, $P_{i,j}(n) = P_{j,i}(n)$, so the total contribution is
\[T(n) = 2 \sum_{i =1}^{\infty} \sum_{j=i+1}^{\infty} P_{i,j}(n) = 2 \sum_{i =1}^{\infty} \sum_{j=i+1}^{\infty} 
\Big( 1+\frac{1-d_{ij}}{2d_{ij}}\left( \frac{2c_{ij}}{1+d_{ij}}\right)^n-\frac{1+d_{ij}}{2d_{ij}}\left( \frac{2c_{ij}}{1-d_{ij}}\right)^n \Big).\]
To analyze this sum, we follow the same three steps as in the proof of Theorem~\ref{Th2}, focussing mainly on the third step, which requires additional work in this case.

\medskip

\textbf{Step 1.} Since $\frac{2c_{ij}}{1+d_{ij}} \leq 2c_{ij}$, the term involving $( \frac{2c_{ij}}{1+d_{ij}} )^n$ can be bounded as follows:
\[\sum_{i =1}^{\infty} \sum_{j=i+1}^{\infty}  \frac{1-d_{ij}}{2d_{ij}}\left( \frac{2c_{ij}}{1+d_{ij}}\right)^n = O \Big( \sum_{i =1}^{\infty} \sum_{j=i+1}^{\infty} (2c_{ij})^n \Big).\]
Since the sum evaluates to
\[\sum_{i =1}^{\infty} \sum_{j=i+1}^{\infty} (2c_{ij})^n = \frac{(2p^2 q)^n}{(1-q^n)(1-q^{2n})}\]
and $2p^2 q \leq \frac8{27}$ for all possible values of $p$ and $q=1-p$, this part is exponentially small and thus negligible.

\medskip

\textbf{Step 2.} The procedure is exactly the same as in the proof of Theorem~\ref{Th2}: we observe that
\[\frac{1+\sqrt{1-4c}}{2\sqrt{1-4c}} = 1 + O(c)\]
as $c \to 0$, thus
\[\frac{1+d_{ij}}{2d_{ij}} = \frac{1+\sqrt{1-4c_{ij}}}{2\sqrt{1-4c_{ij}}} = 1 + O(c_{ij}).\]
Likewise,
\[\frac{2c}{1-\sqrt{1-4c}} = \frac{1+\sqrt{1-4c}}{2} = 1 - c + O(c^2),\]
thus
\[\frac{2c_{ij}}{1-d_{ij}} = 1 - c_{ij} + O(c_{ij}^2).\]
Splitting the sum in a similar way as in the proof of Theorem~\ref{Th2}, namely into the part where $i+j \leq \frac23 \log_{1/q} n$ and the rest, and approximating by means of these Taylor expansions, we find that
\[\sum_{i =1}^{\infty} \sum_{j=i+1}^{\infty} 
\Big( 1-\frac{1+d_{ij}}{2d_{ij}}\left( \frac{2c_{ij}}{1-d_{ij}}\right)^n \Big) = \sum_{i =1}^{\infty} \sum_{j=i+1}^{\infty} \Big( 1 -  \exp \big({- c_{ij} n} \big) \Big)  + o(1).\]

\textbf{Step 3.} The third step is more involved, as we cannot simply plug into a formula taken from the literature. However, we can use a classical technique involving the Mellin transform to analyze the remaining sum. The method is excellently described in \cite{FGD}. From the first two steps, we know that
\[T(n) = 2 \sum_{i=1}^{\infty} \sum_{j=i+1}^{\infty} \Big( 1 -  \exp \big( {- c_{ij} n} \big) \Big)  + o(1).\]

Hence we are now interested in the asymptotic behaviour of the function
\[f(t) = \sum_{i=1}^{\infty} \sum_{j=i+1}^{\infty} \Big( 1 -  \exp \big( {- p^2 q^{i+j-2} t} \big) \Big)\]
as $t \to \infty$. The Mellin transform of this function is given by
\[F(s) = \int_0^{\infty} f(t) t^{s-1}\,dt = -\Gamma(s) \sum_{i=1}^{\infty} \sum_{j=i+1}^{\infty} (p^2 q^{i+j-2})^{-s}
= - \frac{\Gamma(s)}{p^{2s}q^s (1-q^{-s})(1-q^{-2s})}\]
for $s$ in the fundamental strip $\{s \in \mathbb{C}\,:\, -1 < \operatorname{Re} s < 0\}$, in view of the scaling rule
\[\int_0^{\infty} f(a t) t^{s-1}\,dt = a^{-s} \int_0^{\infty} f(t) t^{s-1}\,dt\]
and the fact that
\[\int_0^{\infty} (1-e^{-t})t^{s-1} \,dt = - \Gamma(s)\]
for $\operatorname{Re}(s) \in (-1,0)$.

We can now apply the following general theorem, see \cite[Theorem 4, Corollary 1]{FGD}:

\begin{theorem}\label{FGD-theorem}
Let $f(t)$ be a continuous function on $(0,\infty)$ whose Mellin transform $F(s)$ has the nonempty fundamental strip $\{s \in \mathbb{C}\,:\, \alpha < \operatorname{Re}(s) < \beta\}$. Assume further that $F(s)$ admits a meromorphic continuation to the strip $\{s \in \mathbb{C}\,:\, \alpha < \operatorname{Re}(s) < \gamma\}$ for some $\gamma > \beta$, and that it is analytic on the line $\operatorname{Re}(s) = \gamma$. Moreover, suppose that $F(s)$ satisfies the growth condition
\[F(s) = O(|s|^{-r})\]
for some $r > 1$ uniformly as $|s| \to \infty$, $\eta \leq \operatorname{Re}(s) \leq \gamma$ for some fixed $\eta \in (\alpha,\beta)$ and $|\operatorname{Im}(s)|$ belongs to a denumerable set of values $T_j$ ($T_j \to \infty$).

Then we have
\[f(t) = - \sum_{\xi} \operatornamewithlimits{Res}_{s = \xi} (F(s)t^{-s}) + O(t^{-\gamma})\]
as $t \to \infty$, where the sum is over all poles $\xi$ of $F(s)$ with $\beta \leq \operatorname{Re}(s) \leq \gamma$. 
\end{theorem}

In our case, the poles of the Mellin transform $F(s)$ are located at $0$ (triple pole), the negative integers (which is not relevant for us, since we are only considering those to the right of the fundamental strip) and along the vertical axis $\operatorname{Re}(s) = 0$ at all those values where the denominator vanishes. There are two types: the poles of the form $\chi_k = \frac{2k\pi i}{\log q}$ ($k \in \mathbb{Z} \setminus \{0\}$), where we have a double pole as both $1-q^{-s}$ and $1-q^{-2s}$ vanish, and those of the form $\hat{\chi}_m = \frac{(2m+1)\pi i}{\log q}$ ($m \in \mathbb{Z}$), where we only have a single pole.

Note also that the technical conditions of Theorem~\ref{FGD-theorem} are satisfied: we have $\alpha = -1$, $\beta = 0$, and we can take e.g. $\gamma = 1$ and $\eta = -\frac12$. The fraction $\frac{1}{p^{2s}q^s(1-q^{-s})(1-q^{-2s})}$ remains bounded for $\eta \leq \operatorname{Re}(s) \leq \gamma$ if we stick to horizontal segments of the form $\operatorname{Im}(s) = \frac{(j+1/2)\pi}{\log q}$ with $j \in \mathbb{Z}$, and the gamma function $\Gamma(s)$ is well known to decay at an exponential rate as $|\operatorname{Im}(s)| \to \infty$. Thus the growth condition is satisfied.

It remains to compute the residues at all the aforementioned poles: we first have, with $L = \log_{1/q} t$,
\[\operatornamewithlimits{Res}_{s=0} F(s) t^{-s} = - \frac{L^2}4 + \Big( \frac{\gamma + 2\log p}{2\log q} - \frac14 \Big) L - \frac{6\gamma^2 + \pi^2 + 24\gamma \log p + 24(\log p)^2}{24 (\log q)^2} + \frac{\gamma + 2\log p}{4\log q} + \frac1{24}\]
at the triple pole $0$. Next, for $s = \chi_k = \frac{2k\pi i}{\log q}$,
\[\operatornamewithlimits{Res}_{s=\chi_k} F(s) t^{-s} = \Big({- \frac{\Gamma(\chi_k)L}{2\log q}} + \frac{\Gamma(\chi_k)(4\log p - \log q) - 2\Gamma'(\chi_k)}{4(\log q)^2} \Big) \exp \Big( 2k \pi i L - \frac{4 k\pi i \log p}{\log q} \Big) .\]
Finally, for $s = \hat{\chi}_m = \frac{(2m+1)\pi i}{\log q}$,
\[\operatornamewithlimits{Res}_{s=\hat{\chi}_m} F(s) t^{-s} = \frac{\Gamma(\hat{\chi}_m)}{4\log q} \exp
\Big( (2m+1) \pi i L - \frac{(4m+2)\pi i \log p}{\log q} \Big)\]
Thus Theorem~\ref{FGD-theorem} yields
\begin{align*}
f(t) &= \frac{L^2}4 - \Big( \frac{\gamma + 2\log p}{2\log q} - \frac14 \Big) L + \frac{6\gamma^2 + \pi^2 + 24\gamma \log p + 24(\log p)^2}{24 (\log q)^2} - \frac{\gamma + 2\log p}{4\log q} - \frac1{24} \\
&\quad+ \sum_{k \in \mathbb{Z} \setminus \{0\}} \Big({\frac{\Gamma(\chi_k)L}{2\log q}} - \frac{\Gamma(\chi_k)(4\log p - \log q) - 2\Gamma'(\chi_k)}{4(\log q)^2} \Big) \exp \Big( 2k \pi i L - \frac{4 k\pi i \log p}{\log q} \Big) \\ 
&\quad- \sum_{m \in \mathbb{Z}} \frac{\Gamma(\hat{\chi}_m)}{4\log q} \exp
\Big( (2m+1) \pi i L - \frac{(4m+2)\pi i \log p}{\log q} \Big) + O \Big( \frac1{t} \Big)
\end{align*}
as $t \to \infty$. Since we proved earlier that
\[T(n) = 2f(n) + o(1),\]
we obtain exactly the stated formula.
\end{proof}

Combining the contributions from patterns of the form $ii$ and those of the form $ij$, we obtain the following corollary:

\begin{corollary}\label{combined}
The average total number of distinct adjacent pairs in geometrically distributed words of length $n$ is
\[S(n) + T(n) = \frac12 (\log_{1/q} n)^2 + \Psi_1(\log_{1/q} n) \log_{1/q} n + \Psi_2(\log_{1/q} n) + o(1)\]
for certain periodic functions $\Psi_1$ and $\Psi_2$ of period $2$, which can be obtained by combining the terms of Theorem~\ref{Th2} and Theorem~\ref{Th3}.
\end{corollary}

\nocite{*}
\bibliographystyle{abbrvnat}
\bibliography{sample-dmtcs}
\label{sec:biblio}

\end{document}